\documentclass[a4paper,10pt]{article}
\usepackage[utf8x]{inputenc}

\usepackage{amsmath}
\usepackage{amsfonts}
\usepackage{amssymb}
\usepackage{amsthm} 
\usepackage{graphicx} 
\usepackage[all]{xy} 
\usepackage{a4wide}

\usepackage{setspace}
\newcommand{\Z}{\ensuremath{\mathbb{Z}}}

\newcommand{\N}{\ensuremath{\mathbb{N}}}
\newcommand{\C}{\ensuremath{\mathbb{C}}}

\newcommand{\A}{\ensuremath{\mathcal{A}}}
\newcommand{\B}{\ensuremath{\mathcal{B}}}
\newcommand{\V}{\ensuremath{\mathcal{V}}}
\newcommand{\D}{\ensuremath{\mathcal{D}}}

\newcommand{\F}{\ensuremath{\mathcal{F}}}

\renewcommand{\O}{\ensuremath{\mathcal{O}}}

\newcommand{\MM}{\ensuremath{\mathrm{M}}}

\newcommand{\PP}{\ensuremath{\mathcal{P}}}
\newcommand{\CC}{\ensuremath{\mathcal{C}}}

\newcommand{\op}{\ensuremath{\mathrm{op}}}

\newcommand{\z}{\ensuremath{\mathbf{0}}}
\newcommand{\up}{\ensuremath{{\uparrow\,}}}
\newcommand{\down}{\ensuremath{{\downarrow\,}}}

\theoremstyle{definition} 
\newtheorem{theorem}{Theorem}
\newtheorem{definition}{Definition}
\newtheorem{lemma}{Lemma}
\newtheorem{proposition}{Proposition}
\newtheorem{corollary}{Corollary}

\newtheorem*{convention}{Convention}

\title{Classifying finite-dimensional C*-algebras by posets of their commutative C*-subalgebras}
\author{Bert Lindenhovius}

\begin{document}

\maketitle
\emph{Dedicated to my sister.}
\abstract{We consider the functor $\CC$ that to a unital C*-algebra $A$ assigns the partial order set $\CC(A)$ of its commutative C*-subalgebras ordered by inclusion. We investigate how some C*-algebraic properties translate under the action of $\CC$ to order-theoretical properties. In particular, we show that $A$ is finite dimensional if and only $\CC(A)$ satisfies certain chain conditions. We eventually show that if $A$ and $B$ are C*-algebras such that $A$ is finite dimensional and $\CC(A)$ and $\CC(B)$ are order isomorphic, then $A$ and $B$ must be *-isomorphic.}

\section{Introduction}

Given a C*-algebra $A$ with unit $1_A$, let $\CC(A)$ be the set of all commutative unital C*-subalgebras $C$ of $A$ such that $1_A\in C$. Then $\CC(A)$ becomes a poset\footnote{We refer to the appendix for an overview of the necessary concepts in order theory.} if we order it by inclusion. Now, one could consider the following question: is it possible to recover the structure of $A$ as a C*-algebra from the poset $\CC(A)$? More precisely, if $A$ and $B$ are C*-algebras such that $\CC(A)$ and $\CC(B)$ are isomorphic as posets, can we find an *-isomorphism between $A$ and $B$?

Apart from its mathematical relevance, this problem is of considerable importance for the so called \emph{quantum toposophy} program, where one tries to describe quantum mechanics in terms of topos theory (see e.g., \cite{BI}, \cite{DI}, \cite{HLS}, \cite{Wolters}). In this program, the central objects of research are the topoi $\mathrm{Sets}^{\CC(A)}$ and $\mathrm{Sets}^{\CC(A)^\op}$. The motivation behind at least \cite{HLS} and \cite{Wolters} is Niels Bohr's doctrine of classical concepts, which, roughly speaking, states that a measurement provides a ``classical snapshot of quantum reality''. Mathematically, this corresponds with an element of $\CC(A)$, and knowledge of all classical snapshots should provide a picture of quantum reality, that is as complete as (humanly) possible. The possiblility of reconstructing $A$ from $\CC(A)$ would assure
the soundness of this doctrine.

For commutative C*-algebras, Mendivil showed that the answer to the question is affirmative (see \cite{Mendivil2}). It turns out that more can be said about the commutative case. As we will see, $\CC$ turns out to be a functor from unital C*-algebras to posets. Hence if $f:A\to B$ is a *-isomorphism, then $\CC(f):\CC(A)\to\CC(B)$ is an order isomorphism. In \cite{Hamhalter}, Hamhalter gave not only a different proof of Mendivil's statement that an order isomorphism $\psi:\CC(A)\to\CC(B)$, with $A$ and $B$ commutative C*-algebras, yields an *-isomorphism $f:A\to B$, but he showed as well that this *-isomorphism can be constructed in such a way that $\psi=\CC(f)$. Moreover, he proved that as long as $A$ is not two dimensional, there is only one *-isomorphism that induces $\psi$ in this way.

For non-commutative C*-algebras however, the answer is negative, since Connes \cite{Connes} showed the existence of a C*-algebra $A_c$ (actually even a von Neumann algebra) that is not isomorphic to its opposite algebra $A_c^\op$. Here the opposite algebra is the C*-algebra with the same underlying topological vector space, but with multiplication defined by $(a,b)\mapsto ba$, where $(a,b)\mapsto ab$ denotes the original multiplication. Since $\CC(A)$ is always isomorphic to $\CC(A^\op)$ as poset, for each C*-algebra $A$, the existence of Connes' C*-algebra $A_c$ shows that the order structure of $\CC(A)$ is not always enough in order to reconstruct $A$. More recent counterexamples can be found in \cite{Phillips} and \cite{PV}.

Nevertheless, there are still problems one could study. For instance, in \cite{DH} D\"oring and Harding consider a functor similar to $\CC$, namely the functor $\V$ assigning to a von Neumann algebra $M$ the poset $\V(M)$ of its commutative von Neumann subalgebras, and prove that one can reconstruct the \emph{Jordan} structure, i.e., the anticommutator $(a,b)\mapsto ab+ba$, 
 of $M$ from $\V(M)$. Similarly, in \cite{Hamhalter}, it is shown that if $\CC(A)$ and $\CC(B)$ are order isomorphic, then there exists a quasi-linear Jordan isomorphism between $A_{\mathrm{sa}}$ and $B_{\mathrm{sa}}$, the sets of self-adjoint elements of $A$ and $B$, respectively. Here quasi-linear means linear with respect to elements that commute. In \cite{Hamhalter2}, it is even shown that this quasi-linear Jordan isomorphism is linear when $A$ and $B$ are \emph{AW*-algebras}. 
 
 Moreover, one could replace $\CC(A)$ by a structure with stronger properties. An example of such a structure is an \emph{active lattice}, defined in \cite{HR}, where Heunen and Reyes also show that this structure is strong enough to determine AW*-algebras completely.

In this paper, however, we will only study $\CC$ as a functor and as an invariant for C*-algebras. We shall examine some of its properties, and, as our main result, we shall prove that if $A$ is finite dimensional, $\CC(A)$ has enough structure to determine $A$ up to *-isomorphism. The way this result is proven considerably differs from the methods of Hamhalter and Mendivil in the commutative case, since we cannot use Gel'fand duality and the ensuing topological methods as they do. Our approach also differs from Hamhalter's methodes in the sense that if $\psi:\CC(A)\to\CC(B)$ is an order isomorphism with $A$ finite dimensional, we do not construct a *-isomorphism $f:A\to B$ such that $\CC(f)=\psi$, since it will turn out that this is not always possible. Instead, we look at order-theoretical invariants in $\CC(A)$ that reflect the C*-algebraic invariants in $A$ determinining $A$ up to *-isomorphism.   

The first step is to find order-theoretical properties of $\CC(A)$ that corresponds to $A$ being finite dimensional. These properties turn out to be chain conditions on $\CC(A)$. The next step is to identify $Z(A)$, the center of $A$, as element of $\CC(A)$. This turns out to be the meet of all maximal elements of $\CC(A)$. Finally, we consider the interval $[Z(A),M]$ in $\CC(A)$, where $M$ is any maximal element. This interval is a lattice, which factors into a product of directly indecomposable lattices. The length of the maximal chains in these lattices corresponds exactly to the dimensions of the matrix algebra factors of $A$, since by the Artin-Wedderburn Theorem $A$ factors in a unique way into a product of matrix algebras.

\section{C*-subalgebras of a commutative C*-algebra}

\begin{convention}
All C*-algebras in this article are assumed to be unital.
\end{convention}

\begin{definition}\label{def:subalgebraequivalencerelation}
	Let $A$ be a C*-subalgebra of $C(X)$. Then we define an equivalence relation $\sim_A$ on $X$ by $x\sim_A y$ if and only if $f(x)=f(y)$ for each $f\in A$. We denote the equivalence class of $x$ under this equivalence relation by $[x]_A$. 
\end{definition}

\begin{lemma}\cite[Proposition 5.1.3]{Weaver}\label{lem:subalgofCXisquotientofX}
	A C*-algebra $A$ is a C*-subalgebra of $C(X)$ if and only if there is a compact Hausdorff space $Y$ and a continuous surjection $q:X\to Y$ such that $C_q[C(Y)]=A$. The map $C_q:C(Y)\to C(X)$, $f\mapsto f\circ q$ is an isometric *-homomorphism in this case, and $Y$ and $X/\sim_A$ are homeomorphic.
\end{lemma}

\begin{definition}\label{def:idealalgebra}
	Let $X$ be compact Hausdorff and let $K\subseteq X$ be closed. Then we denote the C*-subalgebra $$\{f\in C(X):f\mathrm{\ is\ constant\ on\ }K\}$$ by $C_K$.
\end{definition}

\begin{lemma}\label{lem:subalgebraisintersectionofidealalgebras}\cite[Proposition 2.2]{Hamhalter}
	Let $A$ be a C*-subalgebra of $C(X)$. Then $A=\bigcap_{x\in X}C_{[x]_A}$, where $[x]_A$ is interpreted as a subset of $X$ (rather than as a point in $X/\sim_A$). 
\end{lemma}
\begin{proof}
	By Lemma \ref{lem:subalgofCXisquotientofX}, we have $A=C_q[C(Y)]$ with $Y=X/\sim_A$ and $q:X\to Y$ the quotient map. So each $f\in A$ is of the form $g\circ q$ for some $g\in C(Y)$, meaning that $f$ is constant on $[x]_A$ for each $x\in X$. Conversely, let $f$ be constant on $[x]_A$ for each $x\in X$. Define $g:Y\to X$ by $g([x]_A)=f(x)$. Then $g$ is well defined, since $f$ is constant on $[x]_A$, and $f=g\circ q$. Moreover, let $U\subseteq\C$ be open. Then $g^{-1}[U]$ is open in $Y$ if and only if $q^{-1}[g^{-1}[U]]$ is open in $X$ by definition of the quotient topology. Since $f=g\circ q$, we find $g^{-1}[U]$ is open in $Y$ if and only if $f^{-1}[U]$ is open in $X$. Since $f$ is continuous, we find that $g^{-1}[U]$ is open, so $g\in C(X)$. Hence $f$ lifts to a function on $C(Y)$, so $f\in A$. Thus
	\begin{align*}
		A & =  \{f\in C(X):f\ \mathrm{is\ constant\ on\ }[x]_A\ \forall x\in X\}\\
		& =  \bigcap_{x\in X}\{f\in C(X):f\ \mathrm{is\ constant\ on\ }[x]_A\}\\
		& =  \bigcap_{x\in X}C_{[x]_A}. \qedhere
	\end{align*}
\end{proof}

\section{The functor $\CC$}

\begin{definition}
	Let $A$ be a C*-algebra with unit $1_A$. We denote the set of its commutative C*-subalgebras containing $1_A$ by $\CC(A)$.
\end{definition}

If we denote the category of unital C*-algebras with unital *-homomorphisms as morphisms by $\mathbf{uCStar}$ and the category of posets with order morphisms as morphisms by $\mathbf{Poset}$, then $$\CC:\mathbf{uCStar}\to\mathbf{Poset}$$ can be made into a functor \cite[Proposition 5.3.3]{Heunen}.
\begin{lemma}\label{lem:CAisinvariantforA}
	$\CC:\mathbf{uCStar}\to\mathbf{Poset}$ becomes a functor if for *-homomorphisms $f:A\to B$ between C*-algebras $A$ and $B$ we define $\CC(f):\CC(A)\to\CC(B)$ by $C\mapsto f[C]$.
\end{lemma}
\begin{proof}
	Let $C\in\CC(A)$. Then the restriction of $f$ to $C$ is a *-homomorphism with codomain $B$. It follows from the First Isomorphism Theorem for C*-algebras (see for instance \cite[Theorem 3.1.6]{Murphy}) that $f[C]$ is a  C*-subalgebra of $B$. Since $f$ is multiplicative, it follows that $f[C]$ is commutative. Clearly $f[C]$ is a *-subalgebra of $B$, so $f[C]\in\CC(B)$. Moreover, we have $f[C]\subseteq f[D]$ if $C\subseteq D$, so $\CC(f)$ is an order morphism. If $f:A\to B$ and $g:B\to D$ are *-homomorphisms, then $\CC(g\circ f)(C)=g\circ f[C]=g[f[C]]=\CC(g)\circ\CC(f)$, and if $I_A:A\to A$ is the identity morphism, then $\CC(I_A)=1_{\CC(A)}$, the identity morphism of $\CC(A)$. Thus $\CC$ is indeed a functor.
\end{proof}

\begin{lemma}\label{lem:CAismeetsemilattice}
	Let $A$ be a C*-algebra. Then $\CC(A)$ has all non-empty meets, where the meet is given by the intersection operator. In particular, $\CC(A)$ is a meet-semilattice, and has a least element $$\C1_A=\{\lambda 1_A:\lambda\in\C\}.$$
\end{lemma}
\begin{proof}
	Elementary.
\end{proof}

\begin{lemma}\cite[Proposition 14]{BH}\label{lem:CAlatticewhenAcommutative}
	Let $A$ be a C*-algebra. Then the following statements are equivalent.
	\begin{enumerate}
		\item[(i)] $A$ is commutative;
		\item[(ii)] $\CC(A)$ is bounded;
		\item[(iii)] $\CC(A)$ is a complete lattice.
	\end{enumerate}
\end{lemma}
\begin{proof}
	This follows immediately from the observation that $A\in\CC(A)$ if and only if $A$ is commutative, and the fact that if $\CC(A)$ has a greatest element, it has all meets, so it must be a complete lattice.
\end{proof}

The next proposition is originally due to Spitters \cite{Spitters}. A similar statement for the functor $\V$ assigning to a von Neumann algebra $M$ the poset $\V(M)$ of its commutative von Neumann subalgebras can be found in \cite{DRSB}.
\begin{proposition}\label{prop:CAisdcpo}
	Let $A$ be a C*-algebra. Then $\CC(A)$ is a dcpo, where $\bigvee\D=\overline{\bigcup\D}$ for each directed $\D\subseteq\CC(A)$, and where $\bigvee\D=\bigcup\D$ if $A$ is finite dimensional.
\end{proposition}
\begin{proof}
	Let $\D\subseteq\CC(A)$ be a directed subset. Let $S=\bigcup\D$. We show that $S$ is a commutative *-algebra. Let $x,y\in S$ and $\lambda,\mu\in\C$, there are $D_1,D_2\in\D$ such that $x\in D_1$ and $y\in D_2$. Since $\D$ is directed, there is some $D_3\in\D$ such that $D_1,D_2\subseteq D_3$. Hence $x,y\in D_3$, whence $\lambda x+\mu y,x^*,xy\in D_3$, and since $D_3$ is commutative, we obtain $xy=yx$. Since $D_3\subseteq S$, it follows that $S$ is a commutative *-subalgebra of $A$. Now, $\overline{S}$ is a commutative C*-subalgebra of $A$, which is the smallest commutative C*-subalgebra of $A$ containing every element of $\D$, hence it is the join of $\D$. If $A$ is finite dimensional, all subspaces of $A$ are closed, hence $\bigvee\D=\overline{\bigcup\D}=\bigcup\D$.
\end{proof}

If $A$ is not finite dimensional, we do not necessarily have the equality $\bigvee\D=\bigcup\D$ for each directed $\D\subseteq\CC(A)$. For instance, let $A=C([0,1])$ and $$\D=\{C_{[0,1/n]}:n\in\N\}.$$ Then $\bigcup\D$ consists of functions in $C([0,1])$ that are all constant on some neighborhood of $0$. Hence $f\notin\bigcup\D$, where $f:[0,1]\to\C$ is defined by $f(x)=x$. However, one can easily show that $\bigcup\D$ is a unital *-subalgebra of $C([0,1])$ that separates all points of $[0,1]$, so $\overline{\bigcup\D}=C([0,1])$ by the Stone-Weierstrass Theorem. Thus $\bigcup\D\neq\bigvee\D$.

\begin{proposition}\label{prop:Coff}
	Let $f:A\to B$ be a *-homomorphism. Then
	\begin{enumerate}
		\item[(i)] If $\CC(f)$ is surjective, then $f$ is surjective;
		\item[(ii)] If $f$ is injective, then $\CC(f)$ is an order embedding such that $$\down\CC(f)[\CC(A)]=\CC(f)[\CC(A)]$$ and 
		\begin{equation}\label{eq:Cfpreservesintersections}
		\CC(f)\left(\bigcap_{i\in I}C_i\right)=\bigcap_{i\in I}\CC(f)(C_i)
		\end{equation}
		for each non-empty subset $\{C_i\}_{i\in I}\subseteq\CC(A)$. If $\{D_j\}_{j\in J}\subseteq\CC(A)$ is a subset such that $\bigvee_{j\in J}D_j$ exists, then $\bigvee_{j\in J}\CC(f)(D_j)$ exists and 
		\begin{equation}\label{eq:Cfpreservesjoins}
		\CC(f)\left(\bigvee_{j\in J}D_j\right)=\bigvee_{j\in J}\CC(f)(D_j).
		\end{equation}
		Moreover $\CC(f)$ has an upper adjoint $\CC(f)_*:\CC(B)\to\CC(A)$ given by $D\mapsto f^{-1}[D]$ such that $$\CC(f)_*\circ\CC(f)=1_{\CC(A)};$$
		\item[(iii)] If $f$ is a *-isomorphism, then $\CC(f)$ is an order isomorphism.
	\end{enumerate}
\end{proposition}
\begin{proof}\
	\begin{enumerate}
		\item[(i)] Assume that $\CC(f)$ is surjective. Let $b\in B$, and let $b_1=\frac{b+b^*}{2}$ and $b_2=\frac{b-b^*}{2i}$. Then $b=b_1+ib_2$, and $b_1$ and $b_2$ are self-adjoint elements of $B$. It follows that $C^*(b_i,1_B)\in\CC(B)$ for each $i=1,2$, hence by the surjectivity of $\CC(f)$, there are $C_1,C_2\in\CC(A)$ such that $\CC(f)(C_i)=C^*(b_i,1_B)$. Since $\CC(f)(C_i)=f[C_i]$, this means that there are $a_1\in C_1$ and $a_2\in C_2$ such that $f(a_i)=b_i$. Let $a=a_1+i a_2$. Then $f(a)=b$, hence $f$ is surjective.
		
		\item[(ii)] Assume that $f$ is injective. We first show that $\CC(f)$ has an upper adjoint $\CC(f)_*$. Let $D\in\CC(B)$ and $x,y\in f^{-1}[D]$. Then $f(x),f(y)\in D$, so $$f(xy-yx)=f(x)f(y)-f(y)f(x)=0.$$ By the injectivity of $f$ it follows that $xy=yx$, so $f^{-1}[D]$ is a commutative *-subalgebra of $A$, which is closed since $f$ is continuous and $D$ is closed. Moreover, since $f(1_A)=1_B$, and $1_B\in D$, it follows that $1_A\in f^{-1}[D]$. So $D\mapsto f^{-1}[D]$ is a well-defined map $\CC(B)\to\CC(A)$. Moreover, this map is clearly inclusion preserving, and since $f[C]\subseteq D$ if and only if $C\subseteq f^{-1}[D]$ for any two subsets $C$ and $D$ of $A$ and $B$, respectively, we find that $D\mapsto f^{-1}[D]$ is indeed the upper adjoint of $\CC(f)$.
		
		Let $\{C_i\}_{i\in I}$ be a non-empty collection of elements of $\CC(A)$. We always have $f[\bigcap_{i\in I}C_i]\subseteq\bigcap_{i\in I}f[C_i]$. Now, let $x\in\bigcap_{i\in I}f[C_i]$. Then for each $i\in I$ there is an $c_i\in C_i$ such that $x=f(c_i)$. Hence for each $i,j\in I$, we have $f(c_i)=f(c_j)$. By injectivity of $f$ it follows that $c_i=c_j$, so $x=f(c)$ with $c\in\bigcap_{i\in I} C_i$ equal to $c_i$ for each $i\in I$. We conclude that $f[\bigcap_{i\in I}C_i]=\bigcap_{i\in I}f[C_i]$, which is exactly (\ref{eq:Cfpreservesintersections}). 
		
		Let $\{D_j\}_{j\in J}$ be a collection of elements of $\CC(A)$ such that $\bigvee_{j\in J}D_j$ exists. Then $\CC(f)(D_k)\leq\CC(f)\left(\bigvee_{j\in J}D_j\right)$ for each $k\in J$. Let $E\in\CC(B)$ such that $\CC(f)(D_k)\leq E$ for each $k\in J$. Since $\CC(f)$ has an upper adjoint $\CC(f)_*$, we find that $D_k\leq\CC(f)_*(E)$ for each $k\in J$, so $\bigvee_{j\in J}D_j\leq\CC(f)_*(E)$. Again using the adjunction, we find $\CC(f)\left(\bigvee_{j\in J}D_j\right)\leq E$. We conclude that $\CC(f)\left(\bigvee_{j\in J}D_j\right)$ is the join of $\{\CC(f)(D_j):j\in J\}$, and it follows automatically that (\ref{eq:Cfpreservesjoins}) holds.
		
		By injectivity of $f$, we find $$\CC(f)_*\circ\CC(f)(C)=f^{-1}[f[C]]=C$$ for each $C\in\CC(A)$, hence $\CC(f)_*\circ\CC(f)=1_{\CC(A)}$. 
		
		Let $C\in\down\CC(f)[\CC(A)]$. Hence there is some $D\in\CC(A)$ such that $C\subseteq\CC(f)(D)$. Then $$\CC(f)_*(C)\subseteq \CC(f)_*\circ\CC(f)(D)=D.$$ Now, $$\CC(f)\circ\CC(f)_*(C)=f[f^{-1}[C]]=C\cap f[A]=C,$$ since $C\subseteq\CC(f)(D)=f[D]\subseteq f[A]$. Thus $C=\CC(f)(E)$ where $E=\CC(f)_*(C)$, hence $C\in\CC(f)[\CC(A)]$. We conclude that $$\down\CC(f)[\CC(A)]=\CC(f)[\CC(A)].$$ 
		
		Finally, we show that $\CC(f)$ is an order embedding. Let $C_1$ and $C_2$ elements of $\CC(A)$. Since $\CC(f)$ is an order morphism, we have $\CC(f)(C_1)\subseteq\CC(f)(C_2)$ if $C_1\subseteq C_2$. Assume now that $\CC(f)(C_1)\subseteq\CC(f)(C_2)$. In other words, we have $f[C_1]\subseteq f[C_2]$. By the injectivity of $f$, we have $f^{-1}[f[C]]=C$ for each $C\in\CC(A)$. Since $f^{-1}$ preserves inclusions, this implies $C_1\subseteq C_2$. Hence we obtain $\CC(f)(C_1)\subseteq\CC(f)(C_2)$ if and only if $C_1\subseteq C_2$. In other words, $\CC(f)$ is an order embedding.

		\item[(iii)] This follows directly from the functoriality of $\CC$ and the fact that $f$ has an inverse.
	\end{enumerate}
\end{proof}

\section{Finite-dimensional C*-algebras}

In this section we shall prove the following theorem.

\begin{theorem}\label{thm:fdmaintheorem}
	Let $A$ be a finite-dimensional C*-algebra and $B$ any C*-algebra such that $\CC(A)\cong\CC(B)$. Then $A\cong B$.
\end{theorem}

The proof of the theorem relies on the fact that every finite-dimensional C*-algebra is *-isomorphic to a finite sum of matrix algebras.
\begin{theorem}[Artin-Wedderburn]\index{Artin-Wedderburn Theorem}
	Let $A$ be a finite-dimensional C*-algebra. Then there are $k,n_1,\ldots,n_k\in\N$ such that $$A\cong\bigoplus_{i=1}^k\MM_{n_i}(\C).$$
	The number $k$ is unique, whereas the numbers $n_1,\ldots,n_k$ are unique up to permutation.
\end{theorem} 
\begin{proof}
	\cite[Theorem I.11.2]{Takesaki1}
\end{proof}

The strategy for proving Theorem \ref{thm:fdmaintheorem} is the following. First, we find an order-theoretic property of $\CC(A)$ that corresponds with the finite-dimensionality of $A$. Then we find a method of retrieving the numbers $k,n_1,\ldots,n_k$ from $\CC(A)$.

\begin{definition}
	Let $\CC$ be a poset. Then $\CC$ is called \emph{Artinian}\index{Artinian poset} if it satisfies one of the following equivalent conditions:
	\begin{enumerate}
		\item[1)]  Every non-empty subset of $\CC$ contains a minimal element;
		\item[2)] All non-empty filtered subsets of $\CC$ have a least element;
		\item[3)] $\CC$ satisfies the \emph{descending chain condition}\index{descending chain condition}: if we have a sequence of
		elements $C_1\geq C_2\geq \ldots$ in $\CC$, i.e., a countable descending chain, then the sequence stabilizes, i.e., there is an
		$n\in\N$ such that $C_k=C_{n}$ for all $k>n$.
	\end{enumerate}
\end{definition}
In order to see that these conditions are equivalent, assume that (1) holds and let $\F$ a non-empty filtered subset of $\CC$. Then $\F$ must have a minimal element $C$. Now, if $F\in\F$, then there must be an $G\in\F$ such that $G\leq C,F$. Since $C$ is minimal, it follows that $G=C$, so $C\leq F$, whence $C$ is the least element of $\F$. 

For (2) implies (3), let $C_1\geq C_2\geq C_3\geq \ldots$ be a descending chain. Then $\F=\{C_i\}_{i\in\N}$ is clearly a directed subset, so it has a least element, say $C_n$. So we must have $C_k=C_n$ for all $k>n$, hence $\CC$ satisfies (3).

Finally, we show by contraposition that (3) follows from (1). So assume that $\CC$ does not satisfy the descending chain
condition. Hence we can construct a sequence $C_1\geq C_2\geq \ldots$ that
does not terminate. The set \mbox{$\F=\{C_n:n\in\N\}$} is then a non-empty
subset of $\CC$ without a minimal element. Thus $\CC$ does not satisfy (1).

There exists also a notion dual to the notion of an Artinian poset.
\begin{definition}
	Let $\CC$ be a poset. Then $\CC$ is called \emph{Noetherian} if is satisfies one of the following equivalent conditions:
	\begin{enumerate}
		\item[1)] Every non-empty subset contains a maximal element;
		\item[2)] All non-empty directed subsets of $\CC$ have a greatest element;
		\item[3)] $\CC$ satisfies the \emph{ascending chain condition}\index{ascending chain condition}: if we have a sequence of
		elements $C_1\leq C_2\leq \ldots$ in $\CC$, i.e., a countable ascending chain, then the sequence stabilizes, i.e., there is an
		$n\in\N$ such that $C_k=C_{n}$ for all $k>n$.
	\end{enumerate}
\end{definition}

The following proposition can be found in \cite{Just} as Theorem 4.21.
\begin{proposition}[Principle of Artinian induction]\index{Artinian induction}\label{prop:ArtinianInduction}
	Let $\CC$ be an Artinian poset and $\PP$ a property such that:
		\begin{enumerate}
			\item $\PP(C)$ is true for each minimal $C\in \CC$ \emph{(induction basis)};
			\item $\PP(B)$ is true for all $B<C$ implies that $\PP(C)$ is true \emph{(induction step)}.
		\end{enumerate}
	Then $\PP(C)$ is true for each $C\in \CC$.
\end{proposition}
\begin{proof}
	Assume that $\F=\{C\in \CC:\PP(C)$ is not true$\}$ is non-empty. Since
	$\CC$ is Artinian, this means that $\F$ has a minimal element $C$.
	Hence $\PP(B)$ is true for all elements $B<C$, so $\PP(C)$ is true by
	the induction step, contradicting the definition of $\F$.
\end{proof}

\begin{definition}\label{def:gradedposet}
	Let $\CC$ be a poset. Then $\CC$ is called \emph{graded}\index{graded poset} if one can define a function $d:\CC\to\N$, called a \emph{rank function}\index{rank function} such that: 
	\begin{enumerate}
		\item[(i)] $d(C)=1$ for each $C\in\min\CC$;
		\item[(ii)] $d(C_1)<d(C_2)$ for each $C_1,C_2\in\CC$ such that $C_1<C_2$;
		\item[(iii)] $C_2$ is a cover of $C_1$ if and only if $C_1\leq C_2$ and $d(C_2)=d(C_1)+1$ for each $C_1,C_2\in\CC$.
	\end{enumerate}
\end{definition}
There is no standard definition of a graded poset. For instance, in \cite{Roman} condition (i) is dropped and $\Z$ is taken as codomain of rank functions. On the other hand, \cite{KRY} assumes condition (i), but not condition (ii). For our purposes, it is convenient to combine both definitions. The next three lemmas are now easy to prove.

\begin{lemma}\label{lem:gradedimpliesartinian}
	Let $\CC$ be a graded poset with rank funcion $d:\CC\to\N$. Then $\CC$ is Artinian. If the range of $d$ is bounded from above, $\CC$ is Noetherian as well.
\end{lemma}
\begin{proof}
	Let $\F\subseteq\CC$ a non-empty filtered subset. Then $d[\F]\subseteq\N$ is non-empty, so it contains a least element $n$. Let $F_1,F_2\in\F$ such that $d(F_1)=d(F_2)=n$. Since $\F$ is filtered, there is an $F\in\F$ such that $F\leq F_1,F_2$. Assume that $F\neq F_1$.  Then $F<F_1$, so $d(F)<d(F_1)=n$ contradicting the minimality of $n$. Hence we must have $F=F_1$ and in a similar way, we find that $F=F_2$. So there is a unique element $F\in\F$ such that $d(F)=n$. Let $F'\in\F$. Since $\F$ is filtered, there is some $F''\in\F$ such that $F''\leq F,F'$. Again if $F''<F$, we find $d(F'')<n$ contradicting the minimality of $n$, so $F''=F$. It follows that $F\leq F'$, so $F$ is the least element of $\F$. We conclude that $\CC$ is Artinian. Now assume that $d[\CC]\subseteq\N$ has an upper bound, then the proof that $\CC$ is Noetherian follows in an analogous way. 
\end{proof}

\begin{lemma}\label{lem:uniquenessrankfunction}
	Let $\CC$ be a graded poset. Then its rank function $d:\CC\to\N$ is unique. 
\end{lemma}
\begin{proof}
By Lemma \ref{lem:gradedimpliesartinian}, $\CC$ is Artinian. Hence every non-empty subset has a minimal element, and in particular $\min\CC\neq\emptyset$. Assume that $g:\CC\to\N$ is a rank function. By definition of a rank function we have $g(M)=d(M)=1$ for each $M\in\min\CC$. 

Let $C\in\CC$ such that $C\notin\min\CC$. Assume that $d(B)=g(B)$ for each $B<C$. The set $\{d(B):B<C\}\subseteq\N$ is non-empty and bounded by $d(C)$, hence it must have a maximum $n$. As a consequence, there is some $B<C$ such that $d(B)=n$. Assume that $n+1\neq d(C)$. Then $C$ does not cover $B$, hence there is some $B'\in\CC$ such that $B<B'<C$. Thus $n=d(B)<d(B')$ contradicting the maximality of $n$. We conclude that $d(B)+1=d(C)$, so $C$ covers some $B$. Hence we obtain $$d(C)=d(B)+1=g(B)+1=g(C),$$ so $d=g$ by Artinian induction.  
\end{proof}

\begin{lemma}\label{lem:gradedposetunderorderiso}
	Let $\phi:\CC\to\D$ be an order isomorphism between graded posets $\CC$ and $\D$ with rank functions $d_\CC$ and $d_\D$, respectively. Then $d_\CC=d_\D\circ\phi$.
\end{lemma}
\begin{proof}
	Let $d=d_\D\circ\phi$. We check that $d$ is a rank function on $\CC$. Let $C\in\min\CC$, then $\phi(C)\in\min\D$, so $d(C)=d_\D\circ\phi(C)=1$.
	
	Let $C_1,C_2\in\CC$ such that $C_1<C_2$. Then $\phi(C_1)\leq\phi(C_2)$, and since $\phi$ is injective, we cannot have equality. Hence $\phi(C_1)<\phi(C_2)$, so $$d(C_1)=d_\D\circ\phi(C_1)<d_\D\circ\phi(C_2)=d(C_2).$$
	
	Finally, let $C_1,C_2\in\CC$. Clearly, $C_2$ covers $C_1$ if and only if $\phi(C_2)$ covers $\phi(C_1)$ if and only if $d_2(\phi(C_2))=d_2(\phi(C_1))+1$ and $C_1\leq C_2$ if and only if $d(C_2)=d(C_1)+1$ and $C_1\leq C_2$. 
	
	Thus $d$ is a rank function on $\CC$ and by Lemma \ref{lem:uniquenessrankfunction}, we obtain $d_1=d$.	
\end{proof}

\begin{lemma}\label{lem:rankfunctiononCA}
	Let $A$ be a finite-dimensional C*-algebra. Then $\CC(A)$ is graded with a rank function $\dim:\CC(A)\to\N$ assigning to each element $C\in\CC(A)$ its dimension. Moreover, the range of $\dim$ is bounded from above.
\end{lemma}
\begin{proof}
Trivial.
\end{proof}

Combining Lemmas \ref{lem:gradedimpliesartinian} and \ref{lem:rankfunctiononCA}, we find that if $A$ is finite dimensional, then $\CC(A)$ is both Artinian and Noetherian. We shall prove that the converse holds as well.

\begin{lemma}\label{lem:maximalelementsofCA}
	Let $A$ be a C*-algebra. Then every element of $\CC(A)$ is contained in a maximal element of $\CC(A)$. In particular, the set $\max\CC(A)$ is non-empty.
\end{lemma}
\begin{proof}
	Let $C\in\CC(A)$ and let $\mathcal{S}=\{D\in\CC(A):C\subseteq D\}$. Then $\mathcal{S}$ is non-empty, so if $\CC(A)$ is Noetherian, we immediately find that $\mathcal{S}$ contains a maximal element. If $\CC(A)$ is not Noetherian, we need Zorn's Lemma. Let $\D=\{D_i\}_{i\in I}$ be a chain in $\mathcal{S}$. Then $\D$ is certainly directed, hence it must have a join $\bigvee\D$ by Proposition \ref{prop:CAisdcpo}. Since $\bigvee\D$ clearly contains $C$, we have $\bigvee\D\in\mathcal{S}$. So for every chain $\CC(A)$, there is an upper bound for the chain in $\mathcal{S}$. By Zorn's Lemma it follows that $\mathcal{S}$ contains a maximal element $M$. Now, $M$ must also be maximal in $\CC(A)$, since if there is some $A\in\CC(A)$ such that $M\subseteq A$, then $C\subseteq A$, so $A\in\mathcal{S}$. By the maximality of $M$ with respect to $\mathcal{S}$, it follows that $A$ must be equal to $M$. 
\end{proof}

\begin{proposition}\label{prop:finitedimensionalmaximalcommutativesubalgebras}
Let $A$ be a C*-algebra and $M$ a maximal commutative *-subalgebra. If $M$ is finite dimensional, then $A$ must be finite dimensional as well.
\end{proposition}
This statement can be found in \cite{KR1} as Exercise 4.12. The solution of this exercise can be found in \cite{KR3}.

In particular, when $A=\MM_{n}(\C)$, there is a nice characterization of the maximal commutative C*-subalgebras of $\CC(A)$.
\begin{lemma}\label{lem:dimensionofmaximalsubalgebraofmatrixalgebra}
	Let $A=\MM_n(\C)$ and $M\in\max\CC(A)$. Then $M$ is $n$-dimensional and there is some $u\in \mathrm{SU}(n)$ such that $$M=\{udu^*:d\in D_n\},$$ where $D_n$ is the commutative C*-subalgebra of $A$ consisting of all diagonal matrices.
\end{lemma}
\begin{proof}
	See \cite[Example 5.3.5]{Heunen}.
\end{proof}

\begin{definition}
 Let $X$ be a topological space with topology $\O(X)$. Then $X$ is called \emph{Noetherian}\index{Noetherian topological space} if the poset $\O(X)$ ordered by inclusion is Noetherian.
\end{definition}

\begin{lemma}\cite[Exercise I.1.7]{Hartshorne}\label{lem:Noetheriancharacterization}
 Let $X$ be a topological space. If $X$ is Noetherian and Hausdorff, then $X$ must be finite.
\end{lemma}
\begin{proof}
First we show that every subset of $X$ is compact. So if $Y\subseteq X$, let $\mathcal U$ be a cover of $Y$. Let $\mathcal V$ be the set of all finite unions of elements of $\mathcal U$. Then $\mathcal V$ covers $Y$ as well, and moreover, $\mathcal V$ is directed. Since $\O(X)$ is Noetherian, $\mathcal V$ has a greatest element $V$. Now, $V$ contains every element of $\mathcal{V}$ and since $\mathcal{V}$ covers $Y$, we find that $V$ must contain $Y$. It follows from the definition of $\mathcal V$ that $V$ can be written as a finite union of elements of $\mathcal U$, so $\mathcal U$ has a finite subcover. Thus $Y$ is compact. Now let $x\in X$. Then $X\setminus\{x\}$ is compact, hence closed. Hence $\{x\}$ is open, and it follows that $X$ is discrete. Since $X$ itself is compact, $X$ must be finite.
\end{proof}

\begin{proposition}\label{prop:ArtinianandNoetherianisfinitedimensional}
 Let $A$ be a C*-algebra. Then $A$ is finite dimensional if and only if $\CC(A)$ is Artinian if and only if $\CC(A)$ is Noetherian.
\end{proposition}
\begin{proof}
 Assume that $A$ is finite dimensional. By Lemma \ref{lem:rankfunctiononCA}, $\CC(A)$ has a rank function whose range is bounded from above. By Lemma \ref{lem:gradedimpliesartinian}, $\CC(A)$ is both Artinian and Noetherian.
 
Assume that $A$ is not finite dimensional. By Lemma \ref{lem:maximalelementsofCA}, $\CC(A)$ has a maximal element $M$. By Proposition \ref{prop:finitedimensionalmaximalcommutativesubalgebras}, it follows that $M$ cannot be finite dimensional. Since $M$ is a unital commutative C*-algebra, the Gel'fand-Naimark Theorem assures that $M=C(X)$ for some compact Hausdorff space $X$, which must have an infinite number of points since $A$ is infinite dimensional.

We construct a descending chain in $\CC(A)$ as follows. First choose a countable subset $\{x_1,x_2,x_3,\ldots\}$ of $X$. Let $$C_n=\{f\in C(X):f(x_1)=\ldots=f(x_n)\}$$ for each $n\in\N$. Clearly, we have $C_1\supseteq C_2\supseteq C_3\supseteq\ldots$. Assume that  $i<j$. Then $\{x_1,\ldots,x_i\}$ and $\{x_j\}$ are disjoint closed sets, hence Urysohn's Lemma assures the existence of some $f\in C(X)$ such that $f[\{x_1,\ldots,x_i\}]=\{1\}$ and $f(x_j)=0$. Clearly, $f\in C_i$, but $f\notin C_j$. This shows that $C_i\neq C_j$, so the chain is descending, but it never stabilizes. 

We construct an ascending chain in $\CC(A)$ as follows. First we notice that since $X$ is infinite and Hausdorff, Lemma \ref{lem:Noetheriancharacterization} implies that $X$ is not Noetherian. So there is an ascending chain $O_1\subseteq O_2\subseteq\ldots$ of open subsets of $X$ that does not stabilize. For each $i\in\N$, let $F_i=X\setminus O_i$. Then $F_1\supseteq F_2\supseteq\ldots$ is a descending chain of closed subsets of $X$, which does not stabilize. For each $i\in I$ let $C_i=C_{F_i}$. Then $C_i$ is a C*-subalgebra of $C(X)$ and if $i\leq j$, we have $F_i\supseteq F_j$, so $C_i\subseteq C_j$. Moreover, if $i<j$ and $F_i\neq F_j$, then there is some $x\in F_i$ such that $x\notin F_j$. By Urysohn's Lemma, there is an $f\in C(X)$ such that $f(x)=0$ and $f(y)=1$ for each $y\in F_j$. Hence $f\in C_j$, but $f\notin C_i$. It follows that $C_i\neq C_j$, so $C_1\subseteq C_2\subseteq\ldots$ is an ascending chain that does not stabilize.

Thus $\CC(A)$ contains an ascending chain as well as a descending chain, neither of which stabilizes. Hence $\CC(A)$ can be neither Noetherian nor Artinian. 
\end{proof}

Recall that the center of a C*-algebra $A$ is the set $$\{x\in A:xy=yx\ \forall y\in A\},$$ which is usually denoted by $Z(A)$.

\begin{lemma}\label{lem:centerisintersectionmaximalcommutatives}
 Let $A$ be a C*-algebra. Then $Z(A)$ is a commutative C*-subalgebra. Moreover, $Z(A)$ is the intersection of all maximal commutative C*-subalgebras of $A$. 
\end{lemma}
\begin{proof}
 For each $y\in A$, consider the map $f_y:A\to A$ given by the assignment $x\mapsto xy-yx$. Clearly this is continuous and linear, so $\ker f_y$ is a closed linear subspace of $A$. Hence $Z(A)=\bigcap_{y\in A}\ker f_y$ is a closed linear subspace as well. If $x,y\in Z(A)$ and $z\in A$, then $$xyz=xzy=zxy,$$ so $xy\in Z(A)$. Moreover, $x^*z=(z^*x)^*=(xz^*)^*=zx^*$, so $x^*\in Z(A)$. Clearly $xy=yx$, and $1_A\in Z(A)$, hence $Z(A)\in\CC(A)$.

Let $x\in\bigcap\max\CC(A)$, i.e., $x\in M$ for each maximal $M\in\CC(A)$. Let $y\in A$. Then $y$ can be written as a linear combination of two self-adjoint elements $a_1,a_2$. If $a\in A$ is self-adjoint, then $C^*(a,1)$ is a commutative C*-subalgebra of $A$ containing $a$. By Lemma \ref{lem:maximalelementsofCA}, it follows that there are $M_1,M_2\in\max\CC(A)$ such that $a_i\in M_i$ for $i=1,2$. Since $x\in M_1,M_2$, it follows that $x$ commutes with both $a_1$ and $a_2$. Hence $x$ commutes with $y$, so $x\in Z(A)$. Thus $\bigcap\max\CC(A)\subseteq Z(A)$.

Now assume that $x\in Z(A)$. Since $x$ commutes with all elements of $A$, it commutes in particular with $x^*$. Hence $x$ is normal. We have $x^*\in Z(A)$ as well, for $Z(A)$ is a *-subalgebra of $A$.  Let $M\in\max\CC(A)$. Then $M\cup\{x,x^*\}$ is a set of mutually commuting elements, which is *-closed and contains $1_A$. It follows that $C^*(M\cup\{x,x^*\})$, the C*-subalgebra of $A$ generated by $M\cup\{x,x^*\}$, is commutative. Since $M$ is maximal, $C^*(M\cup\{x,x^*\})$ must be equal to $M$. As a consequence, $x\in M$, so we find that $x$ is contained in every maximal commutative C*-subalgebra of $A$. Hence $Z(A)\subseteq\bigcap\max\CC(A)$.
\end{proof}

\begin{lemma}\label{lem:centerandproducts}
	Let $A_1,\ldots, A_n$ be C*-algebras. Then $$
	Z\left(\bigoplus_{i=1}^nA_i\right)  =  \bigoplus_{i=1}^nZ(A_i).$$
\end{lemma}
\begin{proof}
This follows directly from the fact that multiplication on $\bigoplus_{i=1}^nA_i$ is calcultated coordinatewisely.
\end{proof}

\begin{lemma}\label{lem:upcenterandproducts}
	Let $A_1,\ldots, A_n$ be C*-algebras. Let $A=\bigoplus_{i=1}^nA_i$ and $C\in\CC(A)$ such that $Z(A)\subseteq C$. Then there are $C_i\in\CC(A_i)$ such that $Z(A_i)\subseteq C_i$ and $C=\bigoplus_{i=1}^nC_i$.
\end{lemma}

\begin{proof}
	Let $p_i:A\to A_i$ be the projection on the $i$-th factor. Then we obtain an order morphism $\CC(p_i):\CC(A)\to\CC(A_i)$. Let $C_i=\CC(p_i)(C)$, or equivalently, $C_i=p_i[C]$. Then $$Z(A_i)=p_i\left[\bigoplus_{i=1}^nZ(A_i)\right]=p_i[Z(A)]\subseteq p_i[C]=C_i$$ 
	for each $i\in I$.
	
	Let $c\in C$, then $p_i(c)\in C_i$ for each $i=1,\ldots,n$, so $$c=p_1(c)\oplus\ldots\oplus p_n(c).$$ Thus $c\in\bigoplus_{i=1}^nC_i$, hence $C\subseteq\bigoplus_{i=1}^nC_i$.
	
	Let $c_1\oplus\ldots \oplus c_n\in\bigoplus_{i=1}^nC_i$. This means that for each $i=1,\ldots,n$ there is a $d^i\in C$ such that $p_i(d^i)=c_i$.
	Here $d^i=d_1^i\oplus\ldots\oplus d_n^i$, with $d_j^i\in A_j$, and in particular we have $d^i_i=c_i$. For each $j=1,\ldots,n$, let $e^j\in A$ be the element $e^j_1\oplus\ldots\oplus e^j_n$, with $$e^j_i=\begin{cases}
	1_{A_i} & i=j;\\
	0_{A_i} & i\neq j.
	\end{cases}$$
	Here $1_{A_i}$ and $0_{A_i}$ denote the unit and the zero of $A_i$, respectively. Since $1_{A_i},0_{A_i}\in Z(A_i)$, Lemma \ref{lem:centerandproducts} assures that $e^j\in Z(A)$. Since $Z(A)\subseteq C$, we find that $e^j\in C$. It follows that $f^j=e^jd^j\in C$. Here $f^j=f^j_1\oplus\ldots\oplus f^j_n$ with $$f^j_i=\begin{cases}
	c_i & i=j;\\
	0_{A_i} & i\neq j.
	\end{cases}$$
	Now, $$f^1+\ldots +f^n=c_1\oplus\ldots\oplus c_n=c,$$ and since $f^j\in C$, it follows that $c\in C$. So $C=\bigoplus_{i=1}^nC_i$.
\end{proof}

\begin{proposition}\label{prop:embeddingofCACBintoCAplusB}
	Let $A_1,\ldots,A_n$ be C*-algebras and let $A=\bigoplus_{i=1}^nA_i$. Define the map $\iota:\prod_{i=1}^n\CC(A_i)\to\CC(A)$ by $\langle C_1,\ldots, C_n\rangle\mapsto C_1\oplus\ldots\oplus C_n$, and let $\pi:\CC(A)\mapsto\prod_{i=1}^n\CC(A_i)$ be the map $\CC(p_1)\times\ldots\times\CC(p_n)$, where $p_i:A\to A_i$ denotes the projection on the $i$-th factor. Thus $\pi$ maps $C\in\CC(A)$ to $\langle p_1[C],\ldots,p_n[C]\rangle$. Then:
	\begin{enumerate}
		\item[(i)] $\iota$ is an embedding of posets;
		\item[(ii)] $\pi$ is surjective;
		\item[(iii)] $\pi\circ\iota=1_{\prod_{i=1}^n\CC(A_i)}$ and $1_{\CC(A)}\leq \iota\circ\pi$;
		\item[(iv)] the restriction of $\iota$ to a map $\prod_{i=1}^n\CC(A_i)\to\up Z(A)$ is an order isomorphism with inverse $\pi$.
	\end{enumerate}
\end{proposition}
\begin{proof}
		For each $i=1,\ldots,n$, let $C_i\in\CC(A_i)$. Then $C_1\oplus\ldots\oplus C_n$ is clearly a commutative C*-algebra of $A$, which is unital since $$1_{A}=1_{A_1}\oplus\ldots\oplus 1_{A_n}.$$ Hence the image of $\iota$ lies in $\CC(A)$, so $\iota$ is well defined. Furthermore, we remark that $\CC(p_i):\CC(A)\to\CC(A_i)$ is an order morphism by Lemma \ref{lem:CAisinvariantforA}.
	\begin{enumerate}
	\item[(i)]  Let $\langle C_1,\ldots, C_n\rangle$ and $\langle D_1,\ldots, D_n\rangle$ be elements of $\prod_{i=1}^n\CC(A_i)$. Then 
	$\langle C_1,\ldots,C_n\rangle\leq \langle D_1,\ldots,D_n\rangle$ implies $C_i\subseteq D_i$ for each $i=1,\ldots, n$. Hence $C_1\oplus\ldots\oplus C_n\subseteq D_1\oplus\ldots\oplus D_n$, which says exactly that 
	 $$\iota(\langle C_1,\ldots, C_n\rangle)\subseteq\iota(\langle D_1,\ldots, D_n\rangle).$$
	 
	 Conversely, if $\iota(\langle C_1,\ldots, C_n\rangle)\subseteq\iota(\langle D_1,\ldots, D_n\rangle)$, we have $$C_1\oplus\ldots\oplus C_n\subseteq D_1\oplus\ldots\oplus D_n.$$ If we let act $\CC(p_i)$ on both sides of this inclusion, we obtain the inclusion $C_i\subseteq D_i$ for each $i=1,\ldots,n$. Hence $$\langle C_1,\ldots, C_n\rangle\leq\langle D_1,\ldots,D_n\rangle.$$ Thus $\iota$ is an embedding of posets.
	 
	 \item[(ii)]
	 Let $\langle C_1,\ldots,C_n\rangle\in\prod_{i=1}^n\CC(A_i)$. If $C=C_1\oplus\ldots\oplus C_n$, then $C\in\CC(A)$ and
	 \begin{eqnarray*}
	 \pi(C) & = & \langle\CC(p_1)(C),\ldots,\CC(p_n)(C)\rangle=\langle p_1[C],\ldots,p_n[C]\rangle\\
	 & = & \langle C_1,\ldots,C_n\rangle.
	\end{eqnarray*}
	 
	\item[(iii)] Let $\langle C_1,\ldots, C_n\rangle\in\prod_{i=1}^n\CC(A_i)$. Let $C=\iota(\langle C_1,\ldots, C_n\rangle)$. Then  $C=C_1\oplus\ldots\oplus C_n$, and by the calculation in (ii), we obtain $\pi(C)=\langle C_1,\ldots, C_n\rangle$. Hence $\pi\circ\iota=1_{\prod_{i=1}^n\CC(A_i)}$.
	
	Let $C\in\CC(A)$. Then $$\iota\circ\pi(C)=\iota(\langle p_1[C],\ldots,p_n[C]\rangle)=p_1[C]\oplus\ldots\oplus p_n[C].$$ 
	Let $c\in C$. Since $C\subseteq A$, and $A=\bigoplus_{i=1}^nA_i$, we have $$c=c_1\oplus\ldots\oplus c_n,$$ with $c_i\in A_i$ for each $i=1,\ldots,n$. Hence $c_i=p_i(c)$, and we find $c=p_1(c)\oplus\ldots\oplus p_n(c)$, so $c\in p_1[C]\oplus\ldots\oplus p_n[C]$. But $$p_1[C]\oplus\ldots\oplus p_n[C]=\iota(\langle p_1[C],\ldots, p_n[C]\rangle)=\iota\circ\pi(C).$$ Hence $c\in\iota\circ\pi(C)$, so $C\subseteq\iota\circ\pi(C)$. We conclude that the inequality $1_{\CC(A)}\leq\iota\circ\pi$ holds.
	
	\item[(iv)] In order to show that $\iota$ restricts to an order isomorphism $$\prod_{i=1}^n\up Z(A_i)\to\up Z(A)$$ with inverse $\pi$, it is enough to show that $\iota\circ\pi(C)=C$ for each $C\in\up Z(A)$. Then the statement follows directly from the equality in (iii). So let $C\in\CC(A)$ such that $Z(A)\subseteq C$. Then Lemma \ref{lem:upcenterandproducts} assures that there are $C_i\in\CC(A_i)$ for each $i=1,\ldots,n$ such that $C=C_1\oplus\ldots\oplus C_n$ and $Z(A_i)\subseteq C_i$ for each $i=1,\ldots,n$. Then $p_i[C]=C_i$, hence 
	\begin{align*}
	\iota\circ\pi(C) &= \iota(\langle p_1[C],\ldots,p_n[C]\rangle)=\iota(\langle C_1,\ldots, C_n\rangle)\\
	& =  C_1\oplus\ldots\oplus C_n=C. \qedhere
	\end{align*}
	\end{enumerate}
\end{proof}

\begin{proposition}\label{prop:intervalZAkommaM}
	Let $A=\bigoplus_{i=1}^k\MM_{n_i}(\C)$, where $k,n_1,\ldots,n_k\in\N$. Then $[Z(A),M]\cong\prod_{i=1}^k\CC(\C^{n_i})$ for each $M\in\max\CC(A)$.
\end{proposition}
\begin{proof}
	By Proposition \ref{prop:embeddingofCACBintoCAplusB}, there is an order morphism $$\pi:\CC(A)\to \prod_{i=1}^j\CC(\MM_{n_i}(\C))$$ whose restriction to $\up Z(A)$ is an order isomorphism with inverse $\iota$. Let $M\in\max\CC(A)$. By Lemma \ref{lem:centerisintersectionmaximalcommutatives} it follows that $Z(A)\subseteq M$, so $M\in\up Z(A)$, hence $M\in\max\up Z(A)$.
	Since $\pi$ is an order isomorphism, it follows that $\pi(M)$ is a maximal element of $\prod_{i=1}^k\CC(\MM_{n_i}(\C))$.
	Clearly there are $M_{n_i}\in\max\CC(\MM_{n_i}(\C))$ for each $i=1,\ldots,k$ such that $\pi(M)= \langle M_{n_1},\ldots, M_{n_k}\rangle$. It follows that $\down \pi(M)=\down M_{n_1}\times\ldots\times\down M_{n_k}$.

Since $\iota$ is the inverse of $\pi$ and has codomain $\up Z(A)$, we find that $$\iota[\down\pi(M)]=\down\iota\circ\pi(M)\cap\up Z(A)=\down M\cap\up Z(A)=[Z(A),M].$$ Hence the restriction $\iota:\down M_{n_1}\times\ldots\times\down M_{n_k}\to[Z(A),M]$ is an order isomorphism.

Notice that all maximal elements of $\CC(\MM_{n_i}(\C))$ are *-isomorphic by Lemma \ref{lem:dimensionofmaximalsubalgebraofmatrixalgebra}. More specificaly, $M_{n_i}\in\max\CC(\MM_{m_i}(\C))$ is *-isomorphic to $D_{n_i}$. Since $$D_{n_i}=\{\mathrm{diag}(\lambda_1,\ldots,\lambda_{n_i}):\lambda_1,\ldots,\lambda_{n_i}\in\C\},$$ we find that $M_{n_i}$ is *-isomorphic to $\C^{n_i}$. Hence there is an embedding $$f:\C^{n_i}\to \MM_{n_i}(\C)$$ such that $f[C^{n_i}]=M_{n_i}$. By Proposition \ref{prop:Coff}, we find that $$\CC(f):\CC(\C^{n_i})\to\CC(\MM_{n_i}(\C))$$ is an order embedding with image $\down M_{n_i}$. Thus there exists an order isomorphism between $\CC(\C^{n_i})$ and $\down M_{n_i}$ in $\CC(\MM_{n_i}(\C))$. Hence $[Z(A),M]\cong\prod_{i=1}^k\CC(\C^{n_i})$.
\end{proof}

\begin{definition}\cite[III.8]{Birkhoff}
	Let $\CC$ be a lattice. Then $\CC$ is called \emph{directly indecomposable}\index{directly indecomposable lattice} if $\CC\cong\CC_1\times\CC_2$ for some lattices $\CC_1,\CC_2$ implies that either $\CC_1=\mathbf{1}$ and $\CC_2=\CC$ or $\CC_1=\CC$ and $\CC_2=\mathbf{1}$, where $\mathbf 1$ denotes the one-point poset.
\end{definition}

The following result is also known as Hashimoto's Theorem.
\begin{theorem}\label{thm:Hashimoto}
	Let $\CC$ be a lattice with a least element $0$. If there are two direct decompositions of $\CC$ 
	\begin{eqnarray*}
		\CC & = & \A_1\times\ldots\times \A_n;\\
		\CC & = & \B_1\times\ldots\times\B_m,
	\end{eqnarray*}
	where $n,m\in\N$, then there are lattices $\CC_{ij}$, $i=1,\ldots,n$ and $j=1,\ldots,m$ such that 
	\begin{eqnarray*}
		\A_{i} & = & \CC_{i1}\times\ldots\times\CC_{im};\\
		\B_j & = & \CC_{1j}\times\ldots\times\CC_{nj}.
	\end{eqnarray*}
\end{theorem}
\begin{proof}
	\cite[Theorem III.4.2]{Gratzer}
\end{proof}

\begin{corollary}\label{cor:Hashimoto}
Let $\A_1\times\ldots\times\A_n=\B_1\times\ldots\times\B_m$, where $n,m\in\N$ and the $\A_i$ and $\B_i$ are directly indecomposable lattices. Then $n=m$, and there is some permutation $$\pi:\{1,\ldots,n\}\to\{1,\ldots,n\}$$ such that $A_i\cong B_{\pi(i)}$ for each $i=1,\ldots,n$.
\end{corollary}

\begin{definition}
	Let $\CC$ be a bounded lattice and $C\in\CC$. Then $D\in\CC$ is called a \emph{complement}\index{complement in a lattice} of $C$ if $C\wedge D=0$ and $C\vee D=1$.
\end{definition}

The next Proposition is actually an application of \cite[Theorem III.4.1]{Gratzer}.
\begin{proposition}\label{prop:directlyindecomposableboundedposet}
	Let $\CC$ be a bounded lattice. If $0$ and $1$ are the only elements of $\CC$ with a unique complement (namely each other), then $\CC$ is directly indecomposable.
\end{proposition}
\begin{proof}
	Let $D\in\CC$ be a complement of $1$. Then $D=1\wedge D=0$. If $D$ is a complement of $0$, then $D=0\vee D=1$. Thus $0$ and $1$ are each other's unique complement. 
	
	Now assume that there exists an order isomorphism $$\phi:\CC_1\times\CC_2\to\CC,$$ where $\CC_1$ and $\CC_2$ are bounded lattices not equal to $\mathbf 1$. This last condition implies that $$\langle 1,1\rangle\neq\langle 1,0\rangle\neq\langle 0,0\rangle.$$ Let $C=\phi(\langle 1,0\rangle)$. Then it follows that $1\neq C\neq 0$. Clearly $\langle 0,1\rangle$ is a complement of $\langle 1,0\rangle$ in $\CC_1\times\CC_2$, but it is also unique. Let $\langle D_1,D_2\rangle$ be a complement of $\langle 0,1\rangle$. Since meets and joins are calculated componentwise, we find
	\begin{eqnarray*}
		\langle D_1,0\rangle & = & \langle D_1\wedge 1,D_2\wedge 0\rangle=\langle D_1,D_2\rangle\wedge \langle 1,0\rangle =\langle 0,0\rangle;\\
		\langle 1,D_2\rangle & = & \langle D_1\vee 1,D_2\vee 0\rangle=\langle D_1,D_2\rangle\vee \langle 1,0\rangle =\langle 1,1\rangle,
	\end{eqnarray*}
	hence $D_1=0$, $D_2=1$. Thus $\langle D_1,D_2\rangle=\langle 0,1\rangle$, whence $\langle 1,0\rangle$ indeed has a unique complement.  Since $\phi$ is an order isomorphism, $\phi$ preserves meets and joins, hence $C$ has a complement $D=\phi(\langle 0,1)\rangle$. Now assume that $C$ has another complement $D'$. Since $\phi$ is an order isomorphism, it follows that $\phi^{-1}(D')$ is a complement of $\langle 1,0\rangle$, and by uniqueness of this complement, we obtain $\phi^{-1}(D')=\langle 0,1\rangle$. We find that $$D=\phi(\langle 0,1\rangle)=\phi\circ\phi^{-1}(D')=D'.$$ The statement follows now by contraposition.
\end{proof}

\begin{proposition}\label{prop:CAisdirectlyindecomposablewhenAfinitedimcomCalg}
	Let $A$ be a commutative finite-dimensional C*-algebra. Then $\CC(A)$ is a directly indecomposable lattice. 
\end{proposition}
\begin{proof}
By Lemma \ref{lem:CAlatticewhenAcommutative}, $\CC(A)$ is a bounded lattice. Let $X$ be the spectrum of $A$.

If $X$ is a singleton set, then $A$ is one-dimensional, hence we have $\CC(A)=\{\C1_A\}$, so $\CC(A)=\mathbf 1$, the one-point lattice, and there is nothing to prove. If $X$ is a two-point set, then $\CC(A)=\{A,\C1_A\}$. So $\CC(A)$ contains no other elements than a greatest and a least one, and is therefore certainly directly indecomposable. 

Assume that $X$ has at least three points. Let $B\in\CC(A)$, assumed not equal to $\C1_A$ or $A$. By Lemma \ref{lem:subalgebraisintersectionofidealalgebras}, we have $B=\bigcap_{x\in X}C_{[x]_B}$. Since $X$ is finite, it follows that $X/\sim_B$ is finite as well. Notice that we cannot have $[x]_B=\{x\}$ for all $x\in X$, otherwise $B=C(X)=A$. Neither can $X/\sim_B$ be a singleton set, since otherwise $B=\C1_A$. For each element $[x]_B$ in $X/\sim_B$, choose a representative $x$. Let $K$ be the set of representatives. Notice that $K$ is not a singleton set, since $X/\sim_B$ contains at least two elements. Also notice that $K$ is not unique, since there is at least one $[x]_B\in X/\sim_B$ containing two or more points. Since $X$ is discrete, it follows that $K$ is closed. 

Let $f\in B\cap C_K$ and let $x,y\in X$ be points such that $x\neq y$. If $[x]_B=[y]_B$, then $f(x)=f(y)$. If $[x]_B\neq [y]_B$, then there are $x',y'\in K$ such that $x'\in [x]_B$ and $y'\in [y]_B$. Since $f\in C_K$, we find that $f(x')=f(y')$. Since $f\in B$, we obtain $f(x)=f(x')$ and $f(y)=f(y')$. Combining all equalities gives $f(x)=f(y)$. So in all cases, $f(x)=f(y)$. So $f$ must be constant, and we conclude that $B\cap C_K=\C1_A$.

Since $\CC(A)$ is a lattice, $B\vee C_K$ exists. Let $f\in C(X)$. Define the map $g:X\to\C$ by $g(x)=f(k)$ if $x\in[k]_B$, where $k\in K$. Notice that is well defined, since $K$ is a collection of representatives. Moreover, since $X$ is discrete, $g$ is continuous, so $g\in C(X)$. By definition, we have $g\in \bigcap_{x\in X}C_{[x]_B}$, so $g\in B$. Let $h=f-g$. Then $h\in C(X)$, and if $k\in K$, we find $h(k)=f(k)-g(k)=0$, so $h$ is constant on $K$. We conclude that $f=g+h$ with $g\in B$ and $h\in C_K$. Hence $A=C(X)=B\vee C_K$.
	
We find that $C_K$ is a complement of $B$. However, $K$ is not unique, and therefore neither is $C_K$. We conclude that $A$ and $\C1_A$ are the only elements with a unique complement, so $\CC(A)$ is indirectly indecomposable. 
\end{proof}
The proof of this proposition is based on the proof of the directly indecomposability of partition lattices in \cite{Sachs}.
More can be said about $\CC(A)$ when $A$ is a commutative C*-algebra of dimension $n$, namely that $\CC(A)$ is order isomorphic to the lattice of partitions of the set $\{1,\ldots,n\}$. We refer to \cite{Heunen2} for a complete characterization of $\CC(A)$ when $A$ is a commutative finite-dimensional C*-algebra.

We are now ready to prove the main result of this section.

\begin{proof}[Proof of Theorem \ref{thm:fdmaintheorem}]
Let $A$ be a finite-dimensional C*-algebra, and $B$ a C*-algebra. Let $\phi:\CC(A)\to\CC(B)$ an order isomorphism. By Proposition \ref{prop:ArtinianandNoetherianisfinitedimensional}, $\CC(A)$ is Noetherian, and so $\CC(B)$ must be Noetherian as well. Hence Proposition \ref{prop:ArtinianandNoetherianisfinitedimensional} implies that $B$ is finite dimensional. 

It follows from Lemma \ref{lem:rankfunctiononCA} that both $\CC(A)$ and $\CC(B)$ have a rank function assigning to each element its dimension. By Lemma \ref{lem:uniquenessrankfunction} the rank function is unique, hence it follows from Lemma \ref{lem:gradedposetunderorderiso} that $\dim(\phi(C))=\dim(C)$ for each $C\in\CC(A)$. Therefore, we can reconstruct the dimensions of elements of $\CC(A)$ and $\CC(B)$, and the dimension is preverved by $\phi$.

By the Artin-Wedderburn Theorem, there are unique $k,k'\in\N$ and unique $\{n_i\}_{i=1}^n,\{n_i'\}_{i=1}^{k'}$ with $n_i,n'_i\in\N$ such that 
\begin{eqnarray*}
A & \cong & \bigoplus_{i=1}^k\MM_{n_i}(\C);\\
B & \cong & \bigoplus_{i=1}^{k'}\MM_{n'_i}(\C).
\end{eqnarray*}
 Without loss of generality, we may assume that the $n_i$ and $n'_i$ form an descending (but not necessarily strictly descending) finite sequence.

By Lemma \ref{lem:centerisintersectionmaximalcommutatives}, we have the equalities $Z(A)=\bigcap\max\CC(A)$ and $\bigcap\max\CC(B)=Z(B)$. Since the intersection is the meet operation in $\CC(A)$ and $\CC(B)$, and order isomorphism preserve both meets and maximal elements, we find that $\phi(Z(A))=Z(B)$, so $\dim(Z(A))=\dim(Z(B))$. Using Lemma \ref{lem:centerandproducts}, we find that $Z(A)=\bigoplus_{i=1}^nZ(\MM_{n_i}(\C))$, and since the dimension of the center of a matrix algebra is $1$, we find that $\dim Z(A)=k$. In the same way, we find that $\dim Z(B)=k'$, so we must have $k=k'$.

Let $M\in\max\CC(A)$. Then $\phi(M)$ is a maximal element of $\CC(B)$, and since $\phi(Z(A))=Z(B)$, we find that $\phi$ restricts to an order isomorphism $[Z(A),M]\to [Z(B),\phi(M)]$. By Proposition \ref{prop:intervalZAkommaM}, we obtain an order isomorphism $$\prod_{i=1}^k\CC(\C^{n_i})\cong\prod_{i=1}^k\CC(\C^{n'_i}).$$

It is possible that for some $i$ we have $n_i=1$, in which case we have $\CC(\C^{n_i})=\mathbf{1}$. Since we assumed that $\{n_i\}_{i=1}^n$ is a descending sequence, there is a greatest number $r$ below $k$ such that $n_r\neq 1$. Likewise, let $s$ be the greatest number such that $n'_{s}\neq 1$. Then we obtain an order isomorphism
$$\prod_{i=1}^{r}\CC(\C^{n_i})\cong\prod_{i=1}^{s}\CC(\C^{n'_i}).$$
By Proposition \ref{prop:CAisdirectlyindecomposablewhenAfinitedimcomCalg} and Corollary \ref{cor:Hashimoto}, we now find $r=s$, and there is a permutation $\pi:\{1,\ldots,r\}\to\{1,\ldots,r\}$ such that $\CC(\C^{n_i})\cong\CC(\C^{n'_{\pi(i)}})$ for each $i\in\{1,\ldots,r\}$.
Let $\psi_i:\CC(\C^{n_i})\to\CC(\C^{n'_{\pi(i)}})$ be the accompanying order isomorphism. Lemma \ref{lem:rankfunctiononCA} assures that the function assigning to each element of $\CC(\C^{n_i})$ its dimension is a rank function, and similarly the dimension function is a rank function for $\CC(\C^{n'_{\pi(i)}})$. By Lemma \ref{lem:gradedposetunderorderiso}, we find that $\dim(C)=\dim(\psi_i(C))$ for each $C\in\CC(\C^{n_{i}})$. Hence $$n_i=\dim(\C^{n_i})=\dim\left(\psi_i\left(\C^{n_i}\right)\right)=\dim\left(\C^{n'_{\pi(i)}}\right)=n'_{\pi(i)},$$ where the fact that order isomorphisms map greatest elements to greatest elements is used in the third equality.

By definition of $r$, we must have $n_i=n_i'=1$ for all $i\geq r$. Hence we can extend $\pi$ to a permutation $\{1,\ldots,k\}\to\{1,\ldots, k\}$ by setting $\pi(i)=i$ for each $i\geq r$. Hence $k=k'$ and $\{n_1,\ldots,n_k\}$ and $\{n_1',\ldots,n'_k\}$ are the same sets up to permutation. We conclude that $A$ and $B$ must be *-isomorphic.
\end{proof}

We note that since the class of all finite-dimensional C*-algebras and the class of all finite-dimensional von Neumann algebras are the same, a similar statement holds for the functor $\V$ assigning to a von Neumann algebra $M$ the poset $\V(M)$ of its commutative von Neumann subalgebras. Thus if $M$ and $N$ are von Neumann algebras such that $M$ is finite-dimensional, then $\V(M)\cong\V(N)$ implies $M\cong N$.

If $A$ is a finite-dimensional C*-algebra and $B$ is a C*-algebra such that there is an order isomorphism $\phi:\CC(A)\to\CC(B)$, then it might be the case that even though $A$ and $B$ are *-isomorphic, we have $\phi=\CC(f)$ for more than one *-isomorphism $f:A\to B$. For instance, let $A=B=\C^2$. Let $f:\C^2\to\C^2$ be given by $f(\langle a,b\rangle)=\langle b,a\rangle$. Then both $\CC(f)=\CC(1_{\C^2})$.

It might even be the case that $\phi\neq\CC(f)$ for each *-isomorphism $f:A\to B$. For instance, let $A=B=\MM_{2}(\C)$. Then $$\CC(\MM_2(\C))=\{\C1_{\MM_{2}(\C)}\}\cup\{uD_2u^*:u\in\mathrm{SU}(2)\},$$ where $D_2=\{\mathrm{diag}(\lambda_1,\lambda_2):\lambda_1,\lambda_2\in\C\}$. Furthermore, one can show that each *-isomorphism $f:\MM_2(\C)\to\MM_2(\C)$ is of the form $a\mapsto uau^{-1}$ for some $u\in \mathrm{U}(2)$ \cite[Theorem 4.27]{AS1}. Hence $\CC(f):\CC(\MM_2(\C))\to\CC(\MM_2(\C))$ is given by $C\mapsto uCu^*$ for some $u\in\mathrm{U}(2)$.
		
	Choose $v\in\mathrm{U}(2)$ such that $D_2\neq vD_2v^*$, and let $$\phi:\CC(\MM_2(\C))\to\CC(\MM_2(\C))$$ be defined by $\phi(D_2)=vD_2v^*$, $\phi(vD_2v^*)=D_2$, and $\phi(C)=C$ for all other $C\in\CC(\MM_2(\C))$. Then $\phi$ is clearly an order isomorphism. However, $\phi\neq\CC(f)$ for each *-isomorphism\\ $f:\MM_2(\C)\to \MM_2(\C)$.

\section{Outlook and subsequent research}
We have shown that $\CC(A)$ is a complete invariant for finite-dimensional C*-algebras, whereas Mendivil and Hamhalter showed that $\CC(A)$ completely determine commutative C*-algebras. The question is whether there are more classes of C*-algebras which can be classified by $\CC(A)$. An interesting class might be that of \emph{AF-algebras}, i.e., C*-algebras $A$ that can be approximated by finite-dimensional C*-algebras. 

Usually one considers only separable AF-algebras, which are C*-algebras $A$ such that $A=\overline{\bigcup_{i=1}^\infty A_i}$, where $A_1\subseteq A_2\subseteq\ldots$ is an ascending chain of finite-dimensional C*-subalgebras of $A$. It is well known that this class of AF-algebras can be classified by Bratteli diagrams \cite{BratteliInductiveLimits} and by K-theory \cite{Elliott}. 

One could also look at C*-algebras $A$ such that $A=\overline{\bigcup\D}$ for some directed set $\D$ consisting of finite-dimensional C*-subalgebras of $A$. In this case $A$ need not be separable, and therefore C*-algebras of these form are called \emph{non-separable} AF-algebras. It turns out that neither Bratteli diagrams nor K-theory can completely classify this class of C*-algebras \cite{FK}, \cite{Katsura}. However, as one might have noticed, the framework of $\CC(A)$ might be suitable in order to classify non-separable AF-algebras if one compares the definition of non-separable AF-algebras with the content of Proposition \ref{prop:CAisdcpo}. If this is indeed the case, then $\CC(A)$ might be an interesting alternative for K-theory. 

Since $\CC(A)$ is a dcpo, and domain theory (see for instance \cite{CLD}) deals with various properties of dcpos, a first step is the study of the domain-theoretical properties of $\CC(A)$. It has been proven in \cite{HL} that $\CC(A)$ is a so-called \emph{algebraic domain} if and only if $A$ is a so-called \emph{scattered} C*-algebra, i.e., a C*-algebra for each every self-adjoint element has a countable spectrum. It might be interesting to compare the domain-theoretical properties of $\CC(A)$ with those of $\V(M)$, the poset $\V(M)$ of commutative von Neumann subalgebras of a von Neumann algebra $M$. For the von Neumann case, we refer to \cite{DRSB}.

Besides AF-algebras and scattered C*-algebras, there are several other classes of C*-algebras that contain the finite-dimensional C*-algebras as subclass, for instance the von Neumann algebras. By Connes's example of a von Neumann algebra non-isomorphic to its opposite \cite{Connes}, there is no hope that Theorem \ref{thm:fdmaintheorem} can be extended to the class of von Neumann algebras, but it turns out that it is possible for the subclas of \emph{type I} von Neumann algebras (all finite-dimensional C*-algebras are type I von Neumann algebras). More generally, $\CC(A)$ determines $A$ for each type I \emph{AW*-algebra}, where we recall that AW*-algebras were introduced by Kaplansky as algebraic generalisations of von Neumann algebras \cite{Kaplansky}. We refer to the author's PhD thesis \cite{MeThesis} for the prove that $\CC(A)$ determines each type I AW*-algebra up to isomorphism, which result we eventually be published in \cite{MeAW}. 

It might be interesting to look at non-unital C*-algebras as well. The reason why we did not consider non-unital C*-algebras lies within quantum toposophy, from which this research evolved. In quantum toposophy one is forced to work constructively; and whereas constructive Gel'fand duality for unital commutative C*-algebras holds (see for instance \cite{BM} and \cite{CS}), it was not known yet whether the non-unital version holds as well. However, Henry recently proved a non-unital version of constructive Gel'fand duality \cite{Henry}, which suggests that non-unital C*-algebras can be incorporated within quantum toposophy as well. 

In the non-unital case one could proceed as follows. If $\mathbf{CStar}$ denotes the category of C*-algebras with *-homomorphisms as morphisms, we can define the functor $\CC_0:\mathbf{CStar}\to\mathbf{Poset}$ as follows. Given a C*-algebra $A$, we denote the poset of commutative C*-algebras by $\CC_0(A)$, and if $f:A\to B$ is a *-homomorphism, $\CC_0(f):\CC_0(A)\to\CC_0(B)$ is defined by $C\mapsto f[C]$. The functor $\CC_0$ shares some properties with $\CC$, for instance Proposition \ref{prop:Coff} holds as well if we replace $\CC$ by $\CC_0$. It is even the case that we can describe injectivity of a *-homomorphism $f:A\to B$ completely in order theoretic properties of $\CC_0(f)$. This is possible, since $\z$, the C*-algebra consisting of only one element $0$, is always an element of $\CC_0(A)$. Hence $f:A\to B$ is injective if and only if $\CC_0(f):\CC_0(A)\to\CC_0(B)$ has an upper adjoint $\CC_0(f)_*:\CC_0(B)\to\CC_0(A)$ such that $\CC_0(f)_*(\z)=\z$. The latter equality translates to $f^{-1}[\{0\}]=\{0\}$, which exactly states that $f$ is injective. 

We expect that Theorem \ref{thm:fdmaintheorem} holds as well if we replace $\CC$ by $\CC_0$. Some minor details in the proofs must be adjusted, but we expect that most lemmas still hold, since each finite-dimensional C*-algebra $A$ is automatically unital, hence $\CC(A)$ can be regarded as subposet of $\CC_0(A)$. 

However, it might be difficult to prove a non-unital version of Mendivil and Hamhalter's theorem to the effect that $\CC_0(A)$ determines a commutative C*-algebra $A$ up to *-isomorphism, since it is desirable that we can identify C*-ideals of $A$ as elements of $\CC_0(A)$ in order to reconstruct $A$, and it is not clear how to make this identification. This is already visible if we consider $\CC_0(\C^2)=\{\z,C_1,C_2,C_3,\C^2\}$, where 
 \begin{eqnarray*}
	C_1 &  = & \{\langle \mu,0\rangle:\mu\in\C\},\\ C_2 & = & \{\langle 0,\nu\rangle:\nu\in\C\};\\
	C_3 & = & \{\langle\lambda,\lambda\rangle:\lambda\in\C\}.
\end{eqnarray*} 
 The least element and the greatest element of $\CC_0(\C^2)$ are $\z$ and $\C^2$, respectively, and $C_1,C_2,C_3$ are mutually incomparable. Here $C_1$ and $C_2$ are the only elements that correspond to ideals of $\C^2$, but it is not possible to distinguish them from $C_3$ in an order theoretical way.

Thus $\CC_0(A)$ has some advantages as well as disadvantages with respect to $\CC(A)$. If $A$ is unital, it could be useful to consider both posets at the same time. In this case, $\CC(A)$ can be considered a subposet of $\CC_0(A)$. It might be interesting to remark that in quantum toposophy, a pair $(\CC,\D)$ of a poset $\CC$ and a subposet $\D$ of $\CC$ exactly corresponds to a \emph{site} $(\CC,J)$, i.e., a poset $\CC$ equipped with a \emph{Grothendieck topology}, such that the category $\mathrm{Sh}(\CC,J)$ of $J$-sheaves is equivalent to $\mathrm{Sets}^{\D^\op}$. Hence if $A$ is unital, then the pair $(\CC_0(A),\CC(A))$ corresponds to a site $(\CC_0(A),J)$ such that $\mathrm{Sh}(\CC_0(A),J)\cong\mathrm{Sets}^{\CC(A)^\op}$. Since one usually studies the topos $\mathrm{Sets}^{\CC(A)^\op}$, it follows that one can integrate $\CC_0(A)$ in an elegant way in the usual framework of quantum toposophy. For details on Grothendieck topologies and sheaves on posets, we refer to \cite{Me}.

\appendix
\section{Order-theoretical notions}
We recall some definitions in order theory and refer to \cite{DP} for a detailed exposition.

A \emph{poset} $(\CC,\leq)$ is a set $\CC$ equipped with a \emph{(partial) order}\index{order!partial} $\leq$. That is, $\leq$ is a binary relation, which is reflexive, antisymmetric and transitive. We often write $\CC$ instead of $(\CC,\leq)$ if it is clear which order is used. A poset $\CC$ becomes a category if we consider its elements as objects, and taking a unique morphism $C\to D$ if and only if $C\leq D$ for each $C,D\in\D$. 

If either $B\leq C$ or $C\leq B$ for each $B,C\in\CC$, we say that $\leq$ is a \emph{linear order}\index{order!linear}, and we call $\CC$ a \emph{linearly ordered set}\index{linearly ordered set}. A linearly ordered subposet of a poset is called a \emph{chain}\index{chain}. Given a poset $\CC$ with order $\leq$, we define the opposite poset $\CC^\op$ as the poset with the same underlying set $\CC$, but where $B\leq C$ if and only if $C\leq B$ in the original order. 

Let $\D\subseteq\CC$ be a subset. Then $\D$ is called an \emph{upper set} or an \emph{up-set} if $C\in\D$ and $D\geq C$ implies $D\in\D$ for each $C,D\in\CC$; a \emph{lower set} or an \emph{down-set} if $C\in\D$ and $D\leq C$ implies $D\in\D$ for each $C,D\in\CC$; \emph{directed} if for each $D_1,D_2\in\D$ there is a $D_3\in\D$ such that $D_1,D_2\leq D_3$; and \emph{filtered} if for each $D_1,D_2\in\D$ there is a $D_3\in\D$ such that $D_1,D_2\geq D_3$.

If $C\in \CC$, we define the up-set and down-set generated by $C$ by $\up C=\{B\in \CC:B\geq C\}$ and $\down C=\{B\in \CC:B\leq C\}$, respectively. We can define the up-set generated by a subset $\D$ of $\CC$ by $\up\D=\bigcup_{D\in\D}\up D$. Similarly, we define the down-set generated by $\D$ by $\down\D=\bigcup_{D\in\D}\down D$. If $B,C\in\CC$, then the set $\{D\in\CC:B\leq D\leq C\}$ is called the \emph{interval}\index{interval} between $B$ and $C$, and is denoted by $[B,C]$. Notice that $[B,C]=\up B\cap\down C$. If $[B,C]=\{B,C\}$, then we say that $C$ \emph{covers an element}\index{covering of an element} $B$, or that $B$ is \emph{covered} by $C$.

Let $\D\subseteq\CC$. Then $D\in\D$ is called a \emph{maximal} element\index{maximal element of a poset} of $\D$ if $\up D\cap\D=\{C\}$; a \emph{minimal} element\index{minimal element of a poset} of $\D$ if $\down D\cap\D=\{C\}$; a \emph{greatest} element\index{greatest element of a poset} of $\D$ if $B\leq D$ for each $B\in\D$; and a \emph{least} element\index{least element of a poset} of $\D$ if $B\geq D$ for each $B\in\D$. Greatest and least elements are always unique. If $\CC$ itself contains a least and a greatest element, usually denoted by $0$ and $1$, respectively, we say that $\CC$ is a \emph{bounded}\index{bounded poset}. The set of all maximal elements of $\CC$ is denoted by $\max\CC$, whereas $\min\CC$ denotes the set of all minimal elements of $\CC$.

If $\D\subseteq\CC$, then an element $C\in\CC$ such that $D\leq C$ for each $D\in\D$ is called an \emph{upper bound}\index{upper bound} of $\D$. Similarly, $C$ is called a \emph{lower bound}\index{lower bound} of $\D$ if $C\leq D$ for each $D\in\D$. If $\D$ has a least upper bound $C$, usually called the \emph{join} of $\D$, then we write $C=\bigvee\D$. Dually, if $C$ is a greatest lower bound of $\D$, usually called the \emph{meet} of $\D$, then we write $C=\bigwedge\D$. If $\D$ is a two-point set, say $\D=\{D_1,D_2\}$, we write $D_1\vee D_2$ instead of $\bigvee\D$, and $D_1\wedge D_2$ instead of $\bigwedge\D$. We say that $D_1\vee D_2$ and $D_1\wedge D_2$ are the \emph{binary join} and \emph{binary meet}, respectively, of $D_1$ and $D_2$. If we consider $\CC$ as a category, then the join of $\D$ is exactly the same as the coproduct of $\D$, whereas the meet of $\D$ is exactly the product of $\D$.

If $\CC$ is a poset such that all binary meets exists, then we call $\CC$ a \emph{meet-semilattice}\index{meet-semilattice}. If the join of all directed subsets of $\CC$ exist, we call $\CC$ a \emph{directed-complete partial order}, abbreviated by \emph{dcpo}\index{directed-complete partial order}\index{dcpo}. If all binary meets and joins exists, we call $\CC$ a \emph{lattice}\index{lattice}. If all arbitrary meets and joins exist, then we call $\CC$ a \emph{complete lattice}\index{lattice!complete}.
Notice that a complete lattice $\CC$ is automatically bounded, since $\bigvee\CC$ is its greatest element, and $\bigwedge\CC$ is its least element. Moreover, if $\CC$ has all arbitrary meets, it is automatically a complete lattice, since the join of a subset $\D$ of $\CC$ is given by $\bigvee\D=\bigwedge\{C\in\CC:D\leq C\ \forall D\in\D\}$.

Let $\CC_1,\CC_2$ be posets and $\phi:\CC_1\to\CC_2$ a map. Then $\phi$ is called an \emph{order morphism}\index{order morphism} if $C\leq D$ implies $\phi(C)\leq \phi(D)$ for each $C,D\in \CC_1$; an \emph{embedding of posets}\index{embedding of posets} if $\phi(C)\leq \phi(D)$ if and only if $C\leq D$ for each $C,D\in\CC_1$; and an \emph{order isomorphism}\index{order isomorphism} if it is an order morphism such that $\phi\circ\psi=1_{\CC_2}$ and $\psi\circ\phi=1_{\CC_1}$ for some order morphism $\phi:\CC_2\to\CC_1$. Here $1_{\CC_{i}}:CC_i\to\CC_i$ is the identity order morphism. If $\phi$ is an order morphism and there is an order morphism $\psi:\CC_2\to\CC_1$ such that for each $C_1\in\CC_1$ and $C_2\in\CC_2$ we have $\phi(C_1)\leq C_2$ if and only if $C_1\leq\psi(C_2)$, we say that $\psi$ is the \emph{upper adjoint} of $\phi$, and $\phi$ the \emph{lower adjoint} of $\psi$. Clearly an embedding of posets $\phi$ is injective, but the converse does not always hold. Moreover, a map $\phi:\CC_1\to\CC_2$ is an order isomorphism if and only if it is a surjective order embedding. If we consider $\CC_1$ and $\CC_2$ as categories, then the upper adjoint is exactly the same as a right adjoint. If $\CC_1$ and $\CC_2$ are both lattices, then $\phi$ is called a \emph{lattice morphism}\index{lattice morphism} if $\phi(C\wedge D) = \phi(C)\wedge\phi(D)$ and $\phi(C\vee D) = \phi(C)\vee\phi(D)$ for each $C,D\in\CC_1$. If $\phi$ is bijective, then $\phi$ is called a \emph{lattice isomorphism}. A lattice morphism $\phi:\CC_1\to\CC_2$ is automatically an order morphism. An order isomorphism between lattices is automatically a lattice isomorphism.

Let $\CC_1,\ldots,\CC_n$ be posets. Then the \emph{cartesian product}\index{cartesian product of posets} of the $\CC_i$ is defined as the set $\prod_{i=1}^n\CC_i$, sometimes also denoted as $\CC_1\times\ldots\times\CC_n$, equipped by the order defined by $$\langle C_1,\ldots, C_n\rangle\leq \langle D_1,\ldots,D_n\rangle$$ if and only if $C_i\leq D_i$ for each $i=1,\ldots,n$. If $\CC_i$ is a lattice for each $i=1,\ldots,n$, then $\prod_{i=1}^n\CC_i$ is a lattice as well.

\section*{Acknowledgements}
The author would like to thank Jonathan Farley, Chris Heunen, Klaas Landsman, Frank Roumen and Sander Wolters for their comments and suggestions. This research has been financially supported by the Netherlands Organisation for Scientific Research (NWO) under TOP-GO grant no. 613.001.013 (The logic of composite quantum systems).

\end{document}